\documentclass{article}
\usepackage{amssymb}
\usepackage{amsfonts}
\usepackage{amsmath}

\newtheorem{theorem}{Theorem}

\newtheorem{conjecture}[theorem]{Conjecture}
\newtheorem{corollary}[theorem]{Corollary}

\newtheorem{definition}[theorem]{Definition}

\newtheorem{lemma}[theorem]{Lemma}

\newtheorem{proposition}[theorem]{Proposition}
\newtheorem{remark}[theorem]{Remark}

\newenvironment{proof}[1][Proof]{\noindent\textbf{#1.} }{\ \rule{0.5em}{0.5em}}

\begin{document}
\title{Self attracting diffusions on a sphere and application to a periodic case\thanks{This work was supported by the Swiss
National Foundation Grant FN $200020\_ 149871/1.$}}
\author{Carl-Erik Gauthier \footnote{Institut de Math\'ematiques, Universit\'e de Neuch\^atel, Rue \'Emile Argand 11, 2000 Neuch\^atel, Switzerland. Email:
    carl-erik.gauthier@unine.ch}}
\maketitle
\begin{abstract}
This paper proves almost-sure convergence for the self-attracting diffusion on the unit sphere $$dX(t)=\sigma  dW_{t}(X(t))-a\int_{0}^{t}\nabla_{\mathbb{S}^n}V_{X_s}(X_t) dsdt,\qquad X(0)=x\in\mathbb{S}^n $$ 
where $\sigma >0$, $a < 0$, $V_y(x)=\langle x,y\rangle$ is the usual scalar product in $\mathbb{R}^n$, and $(W_{t}(.))_{t\geqslant 0}$ is a Brownian motion on $\mathbb{S}^n$. From this follows the almost-sure convergence of the real-valued self-attracting diffusion $$d\vartheta_{t}=\sigma dW_{t}+a\int_{0}^{t}\sin(\vartheta_{t}-\vartheta_{s})dsdt, $$ where $(W_t)_{t\geqslant 0}$ is a real Brownian motion. 
\end{abstract}

\textit{keywords:} reinforced process, self-interacting diffusions,
asymptotic pseudotrajectories, rate of convergence.

MSC 60K35, 60G17, 60J60

\section{Introduction}
In this paper, we are interested in the asymptotic behaviour of the solutions to the stochastic differential equation (SDE)
\begin{equation}\label{eqSphere}
dX(t)=\sigma  dW_{t}(X(t))-a\int_{0}^{t}\nabla_{\mathbb{S}^n}V_{X_s}(X_t) dsdt,\qquad X(0)=x\in\mathbb{S}^n
\end{equation} 
with $\sigma>0$, $a\in\mathbb{R}$, $(W_t(.))_t$ is a Brownian motion on the $n$-dimensional Euclidean unit sphere, $\nabla_{\mathbb{S}^n}$ is the gradient on $\mathbb{S}^n$ and $V_y(x)=\langle x,y\rangle$ where $\langle.,.\rangle$ stands for the canonical scalar product on $\mathbb{R}^{n+1}$.\\
The motivation for investigating (\ref{eqSphere}) lies in the study of the long time behaviour of the real-valued SDE:
\begin{equation}\label{eq}
d\vartheta_{t}=\sigma dW_{t}+a\int_{0}^{t}\sin(\vartheta_{t}-\vartheta_{s})dsdt,\qquad \vartheta_{0}=0,
\end{equation} 
 where $(W_{t})_{t}$ is a real Brownian motion. 

 Identify $\vartheta_{t}$ to $(\cos(\vartheta_{t}),\sin(\vartheta_{t}))\in \mathbb{S}^{1}$ and define $d(x,y)$ as the square of the euclidean distance between $(\cos (x),\sin (x)) $ and $(\cos (y),\sin (y)) $; so that $d(x,y)=2-2\cos(x-y)$. Deriving $d$ with respect to $x$ gives $\partial_{x}d(x,y)=2\sin(x-y)$.\\
 Therefore, depending on the sign of $a$, $a\sin (\vartheta_{t}-y)$ points forward/outward $(\cos (y),\sin (y)) $; that is $(\cos(\vartheta_{t}),\sin(\vartheta_{t}))$ is attracted by $(\cos (y),\sin (y)) $ if $a<0$ and repelled from $(\cos (y),\sin (y)) $ if $a>0$. Hence, intuitively, $(\cos(\vartheta_{t}),\sin(\vartheta_{t}))$ should turn around the circle if $a>0$ and it should converge to some point if $a<0$.  \\
Concerning the case $a>0$, it is proved in the preprint \cite{CEGMB} of the author in collaboration with M.Bena\"im that
\begin{theorem} \label{repulcercle} The law of $$(\cos(\vartheta_{t}),\sin(\vartheta_{t})) $$ converges to the uniform law on the circle. 
\end{theorem}   

This paper intends to prove that the intuition concerning the attractive case (when $a<0$) is true. The first main result of this paper is the following result.
\begin{theorem}\label{MainResult1}
If $a<0$, there exists a random variable $X_{\infty}$ such that $\vert \vartheta_{t}-X_{\infty}\vert = O(t^{-1/2}\log^{\gamma /2} (t)) $, with $\gamma >1$.
\end{theorem}
In 1995, M.Cranston and Y. Le Jan proved a convergence result in \cite{cranston} in the cases where $a\sin(x)$ is replaced by $ax$ (linear case) or $a\times sgn(x)$, with $a<0$ (constant case); this last case being extended in all dimension by O.Raimond in \cite{R2} in 1997. A few years later, S.Herrmann and B.Roynette weakened the condition of the profile function $f$ around $0$ and were still able to get almost sure convergence (see \cite{HerrRoy}) for the solution to the stochastic differential equation
\begin{equation}
d\vartheta_{t}=\sigma dW_{t}+\int_{0}^{t}f(\vartheta_{t}-\vartheta_{s})dsdt.
\end{equation} 
Rate of convergence were given in \cite{HS} by S.Herrmann and M.Scheutzow. For the linear case, the optimal rate is proven and turns out to be $O(t^{-1/2}\sqrt{t})$.
 
However, a common fundamental property of these three papers lies in the fact that the associated profile function $f$ is monotone. 

We point out that self interacting diffusion involving periodicity has already received some attention in 2002 by M.Bena\"im, M.Ledoux and O.Raimond (\cite{BLR}), \textit{but} in the normalized case; that is, when $\int_{0}^{t}\sin(X_{t}-X_{s})ds$ is replaced by $\frac{1}{t}\int_{0}^{t}\sin(X_{t}-X_{s})ds$. The interpretation is therefore different. While the drift term of (\ref{eq}) can be ``seen'' as a summation over $[0,t]$ of the interaction between the current position $X_{t}$ and its position at time $s$ and thus an \textit{accumulation of the interacting force }, their drift is then the \textit{average of the interacting force}. The asymptotic behaviour is then given by the following Theorem. 
\begin{theorem}(Theorem 1.1, \cite{BLR}, Bena\"im, Ledoux, Raimond) \label{thmBLR} Let $(\vartheta_{t})_{t\geqslant 0}$ be a solution to the SDE $$d\vartheta_{t}=dW_{t}+\frac{c}{t}\int_{0}^{t}\sin (\vartheta_{t}-\vartheta_{s})dsdt, $$ with initial condition $\vartheta_{0}=0$ and $c\in\mathbb{R}$. Set $X_{t}=\vartheta_{t} \text{ mod }2\pi \in \mathbb{S}^{1}=\mathbb{R}/2\pi \mathbb{Z}$ and defined the normalized occupation measure
\begin{equation*} \mu_{t}=\frac{1}{t}\int_{0}^{t}\delta_{X_{s}}ds. \end{equation*}
\begin{enumerate}
\item If $c\geqslant -1$, then $\{\mu_{t}\}$ converges almost surely (for the topology of weak* convergence) toward the normalized Lebesgue measure on $\mathbb{S}^{1}\sim [0,2\pi]$, $\lambda(dx)=\frac{dx}{2\pi}$.
\item If $c<-1$, then there exists a constant $\beta (c)$ and a random variable $\varsigma\in [0,2\pi[$ such that $\{\mu_{t}\}$ converges almost surely toward the measure
$$\mu_{c,\varsigma}(dx)=\frac{\exp(\beta (c)\cos(x-\varsigma))}{\int_{\mathbb{S}^{1}}\exp(\beta (c) \cos (y))\lambda(dy)}\lambda (dx). $$
\end{enumerate}
\end{theorem}
An intermediate framework between those considered in Theorem \ref{MainResult1} and Theorem \ref{thmBLR} is to add a time-dependent weight $g(t)$ to the normalized case that increases to infinity when time increases, but ``not too fast''. In that case, O.Raimond proved the following Theorem
\begin{theorem}(Theorem 3.1, \cite{Raimond}, Raimond)\label{thmRai} Let $(\vartheta_{t})_{t\geqslant 0}$ be the solution to the SDE
\begin{equation} d\vartheta_{t}=dW_{t}-\frac{g(t)}{t}\int_{0}^{t}\sin (\vartheta_{t}-\vartheta_{s})dsdt,
\end{equation}
where $g$ is an increasing function such that $\lim_{t\rightarrow}g(t)=\infty$, there exists positives $c,t_0$ such that for $t\geqslant t_0$, $ g(t)\leqslant a\log(t)$ and $\vert g'(t)\vert =O(t^{-\gamma}),$ with $\gamma\in ]0,1]$. Set $X_t=\vartheta_t \text{ mod }2\pi$.

Then, there exists a random variable $X_{\infty}$ in $\mathbb{S}^{1}$ such that almost surely, $\mu_{t}=\frac{1}{t}\int_{0}^{t}\delta_{X_{s}}ds$ converges weakly towards $\delta_{X_{\infty}}$.
\end{theorem} 
Interpreting $\vartheta_t$ as an angle also provide the link between  (\ref{eqSphere}) and (\ref{eq}). Indeed, for $x,y\in\mathbb{S}^{n}$, one has
\begin{equation} \Vert x-y\Vert^2 =2-2\langle x,y\rangle=2-2\cos(D(x,y)),
\end{equation}
where $D(.,.)$ is the geodesic distance on $\mathbb{S}^{n}$. Therefore, Theorem \ref{MainResult1} follows from the more general Theorem
\begin{theorem}\label{MainResult2}
If $a<0$, there exists a random variable $X_{\infty}\in \mathbb{S}^n$ such that $$\Vert X(t)-X_{\infty}\Vert = O(t^{-1/2}\log^{\gamma /2} (t)), $$ with $\gamma >1$ and $\Vert . \Vert$ is the standard Euclidean norm in $\mathbb{R}^{n+1}$.
\end{theorem}
We emphasize that Theorems \ref{repulcercle}, \ref{thmBLR} and \ref{thmRai} are particular cases from more general results proved in the respective papers (Theorem 5 in \cite{CEGMB}, Theorem 4.5 in \cite{BLR} and Theorem 3.1 in \cite{Raimond}). 
\subsection{Reformulation of the problem}
From now on, we assume that $a<0$ and that $n$ is fixed. Since the values of $\sigma$ and $a$ do not play any particular role, we assume without loss of generality that $\sigma=1$ and $a=-1$. Thus (\ref{eqSphere}) becomes 
\begin{equation}\label{eq2}
dX(t)=\sigma dW_{t}(X(t))+\int_{0}^{t}\nabla_{\mathbb{S}^n}V_{X_s}(X_t)dsdt,\qquad X(0)=x\in\mathbb{S}^n
\end{equation}
where $V_y(x)=\langle x ,y\rangle =:V(y,x)$. Since $V$ satisfies Hypothesis 1.3 and 1.4 in \cite{BLR}, then (\ref{eq2}) admits a unique strong solution by Proposition 2.5 in \cite{BLR}. We emphasize that equation (\ref{eq}) admits a unique strong solution because the function $\sin(.)$ is Lipschitz continuous (see for example Proposition 1 in \cite{HerrRoy}).
We shall begin by clarifying the two quantities that appear in (\ref{eq2}). 
Firstly, by Example 3.3.2 in \cite{Hsu}, we have
\begin{equation} dW_{t}(X(t))=dB(t) -X_t\langle X_t ,\circ dB_t\rangle  ,
\end{equation}
where $\circ$ stands for the Stratonovitch integral and $(B(t))_{t\geqslant 0}$ is a $(n+1)$ dimensional Brownian motion.

Secondly, for a function $F: \mathbb{R}^{n+1}\rightarrow \mathbb{R}$, we have
\begin{equation}\nabla_{\mathbb{S}^{n}}(F_{\vert_{\mathbb{S}^{n}}})(x)=\nabla_{\mathbb{R}^{n+1}}F(x)-\langle x, \nabla_{\mathbb{R}^{n+1}}F(x)\rangle x ; \; x\in\mathbb{S}^{n}.
\end{equation}
Hence equation (\ref{eq2}) rewrites
\begin{equation}\label{RenfSphere}
dX (t) =dB (t) -X(t)\langle X (t) ,\circ dB(t)\rangle  + \int_0^t (X_s -\langle X_t, X_s\rangle X_t) dsdt ,
\end{equation}
Following the same idea as in \cite{CEGMB}, we set $U_{t}:=\int_{0}^{t}X(s) ds\in\mathbb{R}^{n+1}$ in order to get the SDE on $\mathbb{S}^n\times \mathbb{R}^{n+1}$:
\begin{equation}
\left\{ 
\begin{array}{l}
dX(t)= dB(t) -X(t)\langle X(t) ,\circ dB(t)\rangle + \sum_{j=1}^{n+1}U_j (t) [e_j- X_j (t)X(t)]dt\medskip \\ 
dU(t)= X(t)dt
\end{array}
\right.  \label{EqDecoupl1}
\end{equation}
with initial condition $(X(0),U(0))=(x,0)$ and $e_{1},\cdots ,e_{n+1}$ stand for the canonical basis of $\mathbb{R}^{n+1}$.

A first property is 
\begin{lemma}\label{normEvol} $d(\Vert U(t)\Vert^{2})=2\langle U(t), X(t)\rangle dt,$ where $\Vert .\Vert $ stands for the standard Euclidean norm in $\mathbb{R}^{n+1}$.
\end{lemma}
\begin{proof} It is a direct consequence of It\^o's Formula.\end{proof}

The paper is organised as follow. In Section \ref{sect2}, we present the detailed strategy used for proving Theorem \ref{MainResult2} whereas the more technical proofs are presented in Section \ref{sect3}.

\section{Guideline of the proof of Theorem \ref{MainResult1}}\label{sect2}
Set $R_t=\Vert U(t)\Vert$ and define $V_t\in\mathbb{S}^{n}$ and $\Theta_t\in [-1,1]$ as follows: 
\begin{equation}
V(t)=\left\{ 
\begin{array}{l}
U(t)/R_t \text{ if } R_t >0\medskip \\ 
X(t) \text{ otherwise }
\end{array}
\right.  
\end{equation}
and 
\begin{equation} \Theta_t=\langle V_t ,X(t)\rangle
\end{equation}

Since $(\int_0^t \langle V(s)-\Theta_s X(s), dB(s))_{t\geqslant 0}$ is a real valued martingale, then by the Martingale representation Theorem by a Brownian motion (see Theorem 5.3 in \cite{Doob}), there exists a real valued Brownian motion $(W_t)_t$ such that 
\begin{equation} \label{martin}
\int_0^t \langle V(s)-\Theta_s X(s), dB(s)\rangle =\int_{0}^t\sqrt{1-\Theta_s^2}dW_s.
\end{equation}
\begin{lemma}  $((\Theta_t,R_t))_{t\geqslant}$ is solution to 
\begin{equation}
\left\{ 
\begin{array}{l}
dY_{t}= \sqrt{1-Y_t^2}dW_{t}+[(r_{t}+\frac{1}{r_{t}})(1-Y_t^2)-\frac{n}{2}Y_t]dt \medskip \\ 
d r_{t}=Y_t dt
\end{array}
\right.  \label{EqPolar}
\end{equation}
whenever $R_t >0$.
\end{lemma}
\begin{proof} From Lemma \ref{normEvol}, we deduce  
\begin{equation}\label{R} dR_t =\Theta_t dt.
\end{equation}
Applying It\^o's Formulae to $\langle U(t), X(t)\rangle $ gives
\begin{eqnarray}\label{UX}
\langle U(t), X(t)\rangle &=& t +\int_0^t (\Vert U(s)\Vert^2-\langle U(s), X(s)\rangle^2)ds+\int_0^t\langle U(s) -\langle X(s) ,U(s)\rangle X(s) ,\circ dB_s\rangle.\notag\\
&=&  \int_0^t (1-\frac{n}{2}\langle X(s),U(s)\rangle ds+\int_0^t (\Vert U(s)\Vert^2-\langle U(s), X(s)\rangle^2)ds\\
& &\qquad + \int_0^t\langle U(s) -\langle X(s) ,U(s)\rangle X(s) , dB_s\rangle \notag
\end{eqnarray}
Since $\langle U(t), X(t)\rangle =R_t\Theta_t$, we have by It\^o's Formulae
\begin{equation} d\Theta_t =\frac{1}{R_t}(d(\langle U(t), X(t))-\Theta_t dR_t).
\end{equation}
Combining (\ref{martin}), (\ref{R}) and (\ref{UX}), we obtain
\begin{eqnarray} d\Theta_t &=&(\frac{1}{R_t}-\frac{n}{2}\Theta_t +R_t (1-\Theta_t^2))dt-\frac{\Theta_t^2}{R_t}dt + \langle V(t)-\Theta_t X(t), dB(t)\rangle \notag\\
&=& [(R_{t}+\frac{1}{R_{t}})(1-\Theta_t^2)-\frac{n}{2}\Theta_t]dt +\sqrt{1-\Theta_t^2}dW_{t}
\end{eqnarray}
\end{proof}
A first important result, whose proof is postponed to Section \ref{sect3}, is

\begin{lemma}\label{rCVGE}
We have $\liminf_{t\rightarrow\infty}\frac{R_{t}}{\sqrt{t}}\geqslant 1$ almost-surely.
\end{lemma}

 From this Lemma, we prove in Section \ref{sect3}
\begin{lemma}\label{Thetacvge}
We have that $(\Theta_t)_{t\geqslant 0}$ and $(\frac{R_{t}}{t})_{t>0}$ converge almost surely to $1$. Furthermore the rate of convergence is $O(t^{-1}\log^\gamma  (t)) $, with $\gamma >1$.
\end{lemma}

From (\ref{R}) and the definition of $U(t)$, it follows from It\^o's Formulae 
\begin{equation}\label{Vprime}
dV_{t}=\frac{1}{R_{t}}(X(t)-\Theta_t V_t)dt.
\end{equation}

\begin{lemma}\label{Vcvge} $V_{t}$ converges almost surely.
\end{lemma}
\begin{proof} 
Since 
\begin{equation}\label{diff} \Vert X(t)-\Theta_t V_t \Vert =\sqrt{1-\Theta_t^2},\end{equation}
it follows from (\ref{Vprime}) that 
\begin{equation}\label{viteV} \frac{1}{R_t}\Vert X(t)-\Theta_t V_t \Vert =O(t^{-3/2}\log^{\gamma /2} (t)),
\end{equation}
which is an integrable quantity. 
\end{proof}

We can now prove the main result.

\textbf{Proof of Theorem \ref{MainResult1}.}\\
By Lemmas \ref{Thetacvge} and \ref{Vcvge}, there exists a random variable $X_\infty \in\mathbb{S}^n$ such that $\lim_{t\rightarrow\infty} \Theta_t V_t =X_\infty$. 
The rate of convergence follows from the triangle inequality, (\ref{diff}), (\ref{viteV}) and Lemma \ref{Thetacvge}.
\section{Proofs of Lemma \ref{rCVGE} and Lemma \ref{Thetacvge}}\label{sect3}
\subsection{Proof of Lemma \ref{rCVGE}.}
Set $M_t=-2\int_0^t\langle U(s) -\langle X(s) ,U(s)\rangle X(s) , dB_s\rangle$. Then its quadratic variation is
\begin{eqnarray}
\langle M\rangle_t &=& 4\int_0^t \Vert  U(s) -\langle X(s) ,U(s)\rangle X(s) \Vert^2 ds. \notag\\
&=& 4\int_0^t(\Vert U(s)\Vert^2-\langle U(s), X(s)\rangle^2)ds .
\end{eqnarray} 
So, from (\ref{UX}) and Lemma \ref{normEvol}, we obtain
\begin{equation}\exp(-2(\langle X(t) ,U(t)\rangle+\frac{1}{4}R_t^2)) = \exp(-2t)\exp(M_{t}-\frac{1}{2}\left\langle M_{t}\right\rangle),
\end{equation}

Since
 \begin{equation}\left\langle M\right\rangle_{t}\leqslant 4\int_{0}^{t}\Vert U(s)\Vert^2 ds \leqslant  8\int_{0}^{t}\Vert U(0)\Vert^2+s^2 ds=8(\frac{t^3}{3}+\Vert U(0)\Vert^2 t) ,\end{equation}
 $M_{t}$ satisfies the Novikov Condition (see \cite{K-S}, Chapter V, section D, page 198). Therefore $$\mathcal{E}_{M}(t):=\exp(M_{t}-\frac{1}{2}\left\langle M\right\rangle_{t}) $$ is a martingale having 1 as expectation.

Thus
\begin{equation}\mathbb{E}(\exp(-2(\langle X(t) ,U(t)\rangle+\frac{1}{4}R_t^2))=\exp(-2t).
\end{equation}
Consequently, $$\exp(-2(\langle X(t) ,U(t)\rangle+\frac{1}{4}R_t^2))=\exp(-2(R_{t}\Theta_t+\frac{1}{4}R_{t}^{2}))$$ converges almost surely to $0$ by the Markov inequality and the Borel-Cantelli Lemma. Furthermore, for all $0<c<2$, there exists a random variable $T$ such that almost surely  
\begin{equation} \exp(-2(R_{t}\Theta_t+\frac{1}{4}R_{t}^{2}))\leqslant \exp(-ct),\qquad \forall t\geqslant T.
\end{equation}
Hence, for $t\geqslant T$,
$$ R_{t}\Theta_t+\frac{1}{4}R_{t}^{2}\geqslant \frac{c}{2} t. $$
Therefore, by choosing $c=1$, we obtain 
\begin{equation} R_{t}\geqslant -1+\sqrt{t} \text{ for }t\geqslant T.\end{equation} 
This concludes the proof. 

\subsection{Proof of Lemma \ref{Thetacvge}.}
Before starting the proof of Lemma \ref{Thetacvge}, let us recall the Definition of an \textit{asymptotic pseudotrajectory} introduced by Bena\"im and Hirsch in \cite{BH}.
\begin{definition}\label{defpseudoasympt} Let $(M,d)$ be a metric space and $\Phi$ a semiflow; that is $$\Phi:\mathbb{R}_{+}\times M\rightarrow M:(t,x)\mapsto \Phi(t,x)=\Phi_{t}(x)$$ is a continuous map such that $$\Phi_{0}=Id \text{ and } \Phi_{t+s}=\Phi_{t}\circ\Phi_{s}$$ for all $s,t\in\mathbb{R}_{+}$.\\
A continuous function $X:\mathbb{R}_{+}\rightarrow M$ is an \textit{asymptotic pseudotrajectory for $\Phi$ } if
\begin{equation}\label{pseudoAsympt} \lim_{t\rightarrow\infty}\sup_{0\leqslant h\leqslant T}d(X(t+h),\Phi_{h}(X(t)))=0
\end{equation} 
for any $T>0$. In words, it means that for each fixed $T>0$, the curve $X:[0,T]\rightarrow M:h\mapsto X(t+h)$ shadows the $\Phi$-trajectory over the interval $[0,T]$ with arbitrary accuracy for sufficiently large $t$.\\
If $X$ is a continuous random process, then $X$ is an \textit{almost-surely asymptotic pseudotrajectory for $\Phi$ } if (\ref{pseudoAsympt}) holds almost-surely.
\end{definition}
\begin{theorem} (Theorem 1.2 in \cite{BH})\label{rappel} Suppose that $X([0,\infty))$ has compact closure in $M$ and set $L(X)=\bigcap_{t\geqslant 0}\overline{X([t,\infty))} $. Let $A$ be an attractor for $\Phi$ with basin $W$. If $X_{t_k}\in W$ for some sequence $t_k\rightarrow\infty$, then $L(X)\subset A$. 
\end{theorem}
The following result gives a sufficient condition between a SDE on $\mathbb{R}$ and the related ODE when the diffusion term vanishes.
\begin{theorem}(Proposition 4.6 in \cite{BH})\label{CRITasymp} Let $g:\mathbb{R}\rightarrow\mathbb{R}$ be a Lipschitz function and $\sigma :\mathbb{R}_{+}\times \mathbb{R}\rightarrow \mathbb{R}$ a continuous function. Assume there exists a non-increasing function $\varepsilon : \mathbb{R}_{+}\rightarrow \mathbb{R}_{+}$ such that $\sigma^{2}(t,x)\leqslant \varepsilon (t)$ for all $(t,x)$ and such that 
\begin{equation}\label{AsympCrit} \forall k>0,\; \int_{0}^{\infty}\exp(-k/\varepsilon (t))dt<\infty \footnote{For example $\varepsilon (t)=O(1/(\log (t))^{\alpha})$ with $\alpha >1$.}. \end{equation} 
Then, all solution of $$dx_{t}=g(x_{t})dt +\sigma (t,x_t)dB_{t}$$ is with probability 1 an asymptotic pseudotrajectory for the flow induced by the ODE $\dot{X}(t)=g(X(t))$. Furthermore, for all $T>0$, there exist constant $C,C(T)>0$, such that for all $\beta >0$,
\begin{equation}\label{vit} \mathbb{P}(\sup_{0\leqslant h\leqslant T}\vert x_{t+h}-\Phi_{h}(x_{t})\vert\geqslant \beta)\leqslant C\exp(-(\beta C(T))^2/\varepsilon (t)).\end{equation}
\end{theorem} 
\begin{remark}\label{perturbdeterm} The same result holds if $(x_t)_{t\geqslant 0}$ is solution to the SDE $$dx_{t}=g(x_{t})dt +\sigma (t,x_t)dB_{t}+ \delta (t)h(x_{t})dt, $$ where $h$ is a bounded function and $\delta$ is a non-negative function with $\lim_{t\rightarrow \infty}\delta (t)=0$. In that  case, for all $T>0$, there exist constant $C,C(T,h)>0$, such that for all $\beta >0$,
\begin{equation}\label{vit2} \mathbb{P}(\sup_{0\leqslant h\leqslant T}\vert x_{t+h}-\Phi_{h}(x_{t})\vert\geqslant \beta)\leqslant C\exp(-(\beta-\sup_{s\in [t,t+T]}\delta (s))^2 C(T,h)/\varepsilon (t)).\end{equation}
\end{remark}
\textbf{Proof of Lemma \ref{Thetacvge}:}

The proof is divided into two parts.

\textbf{\textit{Proof of the convergence:}}

First we prove that $\Theta_t$ converges almost surely to $1$. Recall that 
\begin{equation} d\Theta_t = \sqrt{1-\Theta_t^2}dW_{t}+[(R_{t}+\frac{1}{R_{t}})(1-\Theta_t^2)-\frac{n}{2}\Theta_t]dt.
\end{equation}

Define $\alpha(t)=(\frac{3}{2}t)^{\frac{2}{3}}$ so that $\dot{\alpha}(t)=\alpha^{-\frac{1}{2}}(t)$. Set $Z_{t}:=\Theta_{\alpha (t)}$ and $M_{t}=W_{\alpha (t)}$. Thus $(M_{t})_{t}$ is a martingale with respect to the filtration $\mathcal{G}_{t}=\sigma\{ W_{s}\; \vert\; 0\leqslant s\leqslant \alpha (t)\}$, whose quadratic variation at time $t$ is $\alpha(t)=\int_{0}^{t}(\sqrt{\dot{\alpha}(s)})^{2}ds$.\\ Then by the Theorem of Martingale Representation by a Brownian motion (see Theorem 5.3 in \cite{Doob}), there exists a Brownian motion $(B_{t}^{(\alpha)})_{t} $ adapted to $(\mathcal{G}_{t})_{t}$ such that $$M_{t}=\int_{0}^{t}\sqrt{\dot{\alpha} (s)}dB_{s}^{(\alpha)}. $$
Then 
\begin{eqnarray}
Z_{t}&=& \int_{0}^{\alpha (t)}\sqrt{1-\Theta_s^2}dW_{s}+\int_{0}^{\alpha (t)}(R_{s}+\frac{1}{R_s})(1-\Theta_s^2)ds-\frac{n}{2}\int_{0}^{\alpha (t)}\Theta_s ds\\
&=& \int_{0}^{t}\sqrt{\dot{\alpha} (s)}\sqrt{1-Z_s^2}dB_{s}^{(\alpha)}-\int_{0}^{t}\frac{R_{\alpha(s)}+\frac{1}{R_{\alpha(s)}}}{\sqrt{\alpha(s)}}ds-\frac{n}{2}\int_{0}^{t}\dot{\alpha}(s)Z_s ds, \notag
\end{eqnarray}
For $y\in [-1,1]$, let $(Y_t^y)_{t\geqslant 0}$ be the solution to the SDE taking values in $[-1,1]$
 \begin{equation}
\left\{ 
\begin{array}{l}
dY_{t}^y= \sqrt{\dot{\alpha}(t)}\sqrt{1-(Y_t^y)^2}dB_{t}^{(\alpha)}+[\frac{1}{2}(1-(Y_t^y)^2)-\frac{n}{2}\dot{\alpha}(t)Y_t^y]dt \medskip \\ 
Y_0^y=y
\end{array}
\right.  \label{Ysyst}
\end{equation} 
We divide the proof of the convergence in two steps. In the first one, we prove for all $y\in [-1,1]$, $Y_t^y$ converges almost surely to $1$; and then prove the convergence of $Z_t$ to $1$ in the second one.

\textit{Step I:} Let $y\in [-1,1]$. In order to lighten the notation, we omit the superscript $y$ in $Y_t^y$ in this step. We start by proving that $Y_t$ is an asymptotic pseudotrajectory for the flow induced by the ODE  
\begin{equation}\label{limODE} \dot{x}=\frac{1}{2}(1-x^2).
\end{equation} 
In order to achieve it, we use Theorem \ref{CRITasymp}. Since $x\mapsto (1-x^2)$ is Lipschitz continuous on $[-1,1]$ and that $Z_t\in [-1,1]$ for all $t\geqslant 0$, it remains to prove the hypothesis concerning the noise term.

Set 
\begin{equation}
\varepsilon (t):=\dot{\alpha}(t)=(\frac{3}{2}t)^{-\frac{1}{3}}.
\end{equation}
 We have to show that $\varepsilon (t)$ satisfies (\ref{AsympCrit}). But this is immediate because for all $k>0$  
\begin{equation}\int_{0}^{\infty}\exp(-k t^{1/3})dt<\infty .\end{equation}

Since $Y_t\in [-1,1]$ for all $t\geqslant 0$, it is clear that the condition in Remark \ref{perturbdeterm} is satisfied. 
Consequently, by Theorem \ref{CRITasymp}, $(Y_{t})_{t}$ is an asymptotic pseudotrajectory of (\ref{limODE}). \\

Because $\{1\}$ is an attractor for the flow induced by (\ref{limODE}) with basin $]-1,1]$ and that almost-surely $Y_t\in ]-1,1]$ infinitely often, then 
\begin{equation} \lim_{t\rightarrow\infty}Y_t = 1 \text{ a.s.} 
\end{equation}

\textit{Step II:} Our goal is to prove 
\begin{equation}\mathbb{P}(\lim_{t\rightarrow\infty}Z_t =1)=1.\end{equation}
Define the stopping times $\tau_0=0$, 
\begin{equation}\tau_j=\inf (t>\sigma_j\mid \frac{R_{\alpha(t)}}{\sqrt{\alpha(t)}}=\frac{1}{2}),\quad j\geqslant 1\end{equation}
and 
\begin{equation}\sigma_j=\inf (t>\tau_{j-1}\mid \frac{R_{\alpha(t)}}{\sqrt{\alpha(t)}}=\frac{3}{4}),\quad j\geqslant 1.\end{equation}
By Lemma \ref{rCVGE}, we have 
\begin{equation}\label{fixation}\mathbb{P}(\bigcup_{j\geqslant 1}\{\tau_j =\infty\})=1 \text{ and } \mathbb{P}(\sigma_j <\infty\mid\tau_{j-1}<\infty)=1.\end{equation}
Let start by computing $\mathbb{P}(\lim_{t\rightarrow\infty}Z_t =1,\; \tau_j=\infty)$. For $s\in[\sigma_j ,\tau_j]$, we have $$\frac{R_{\alpha(s)}+\frac{1}{R_\alpha(s)}}{\sqrt{\alpha(s)}}\geqslant \frac{1}{2}. $$
So, by a comparison result (see Theorem 1.1, Chapter VI in \cite{IW}),
\begin{equation} \mathbb{P}(Z_{(t+\sigma_j)\wedge \tau_j}\geqslant Y_{(t+\sigma_j)\wedge \tau_j}^{Z_{\sigma_j}},\; \forall t\geqslant 0)=1.\end{equation}
As a consequence, we have
\begin{eqnarray}\mathbb{P}(\lim_{t\rightarrow\infty}Z_t =1,\; \tau_j=\infty,\; \sigma_j<\infty)&\geqslant& \mathbb{P}(\lim_{t\rightarrow\infty}Y_{t+\sigma_j}^{Z_{\sigma_j}} =1,\; \tau_j=\infty,\; \sigma_j<\infty)\notag\\
&=&\mathbb{P}(\tau_j=\infty,\; \sigma_j<\infty).\end{eqnarray}
where the last equality follows from Step I.
Because $\{\sigma_{j} =\infty\}\subset \{\tau_{j} =\infty\}$, it follows from (\ref{fixation})
\begin{equation}\mathbb{P}(\lim_{t\rightarrow\infty}Z_t =1,\; \tau_j=\infty)=\mathbb{P}(\lim_{t\rightarrow\infty}Z_t =1,\; \tau_{j-1}=\infty)+\mathbb{P}(\tau_j=\infty,\; \sigma_j<\infty).\end{equation}
Thus,
\begin{eqnarray} \mathbb{P}(\lim_{t\rightarrow\infty}Z_t =1,\; \tau_j=\infty)&=&\sum_{k=1}^j \mathbb{P}(\tau_k=\infty,\; \sigma_k<\infty)\notag\\
&=& \mathbb{P}(\tau_j=\infty).
\end{eqnarray}
Since $(\{\tau_j =\infty\})_{j\geqslant 0}$ is an increasing family of event, we obtain from (\ref{fixation})
\begin{eqnarray}\mathbb{P}(\lim_{t\rightarrow\infty}Z_t =1)&=&\lim_{j\rightarrow\infty}\mathbb{P}(\lim_{t\rightarrow\infty}Z_t =1,\; \tau_j=\infty)\notag\\
&=& \lim_{j\rightarrow\infty}\mathbb{P}(\tau_j=\infty)\notag\\
&=& 1.\end{eqnarray}
Consequently, $\Theta_t$ converges a.s to $1$.
Therefore,  
\begin{equation}\label{asympr}\frac{R_t}{t}=\frac{1}{t}(R_{0}+\int_{0}^t \Theta_s ds) \end{equation}
converges a.s to $1$.

\textbf{\textit{Proof of the rate of convergence:}}

In view of the previous part, it suffices to determine the rate of convergence to $1$ of $Y_t^y$, where $\alpha(t)$ is now $\alpha (t)=\sqrt{2t}$. 

For $c> 0$, let $(Y_c^y (t))_{t\geqslant 0}$ be the solution to the SDE
 \begin{equation}
(SDE(c,y))\left\{ 
\begin{array}{l}
dY_{c}^y (t)= \sqrt{c}\sqrt{\dot{\alpha}(t)}\sqrt{1-(Y_c^y(t))^2}dB_{t}^{(\alpha)}+[\frac{1}{2}(1-(Y_c^y(t))^2)-\frac{n}{2}\dot{\alpha}(t)Y_c^y(t)]dt \medskip \\ 
Y_c^y(t)=y
\end{array}
\right.  \label{Ysyst2}
\end{equation}
with $\alpha (t)=\sqrt{2t}$. Note that $Y_t^y$ is a strong solution to $SDE(1,y)$. Let $(\varphi (t))_{t\geqslant 0}$ be the solution to the SDE
 \begin{equation}
\left\{ 
\begin{array}{l}
d\varphi (t)= \sqrt{n\dot{\alpha}(t)}d\bar{B}_{t}-\frac{1}{2}\sin (\varphi (t))dt \medskip \\ 
\varphi (0)=\varphi_0
\end{array}
\right.  \label{phisyst}
\end{equation}
where $(\bar{B}_t)_t$ is a real Brownian motion. Then $(\cos(\varphi (t)))_{t\geqslant 0}$ is a weak solution to $SDE(n,\cos(\varphi_0))$. So, if $\vert \varphi (t)\vert =O(\Delta (t))$, then by a Taylor expansion, $1-\cos(\varphi (t))=O(\Delta^2(t))$. 

Since for all $c>0$, $Y_{1}^y (t)$ and $Y_c^y(t)$ converge to $1$ with a similar rate (by using the fact that $c\dot{\alpha}(t)=\dot{\alpha}(\frac{t}{c^2})$), it suffices to find out the rate of $Y_n^y (t)$.\\   
The goal now consists on identifying such a function $\Delta (.)$.
We proceed in three steps. First, we assume that there exists a positive decreasing function $\delta$ that converges to $0$ such that
\begin{enumerate}
\item[i.] for all $\nu>0$ and $0<\kappa <1$, $\lim_{t\rightarrow \infty} \delta (t)\exp(\nu t)=\infty$ and there exists $0<c(\kappa)<\infty$ such that $\lim_{t\rightarrow \infty}\frac{\delta(\kappa t)}{\delta (t)}=c(\kappa)$,\\
\item[ii.] for all $T>0$, $\sup_{0\leqslant h\leqslant T}\vert \varphi (t+h)-\Phi_{h}(\varphi (t))\vert =O(\delta (t))$, where $\Phi$ is the flow induced by the ODE $\dot{x}=-\frac{1}{2}\sin (x)$
\end{enumerate}
and prove that $\vert Z_{t}\vert =O(\delta (t))$.
Next, we prove the existence of such a function $\delta$ and conclude the proof in the last step.
\bigskip

\textit{Step I:} Let $T>0$. From the convergence of $\cos(\varphi(t,\omega))$ to $1$, there exists $t_1(\omega)$ such that, without loss of generality, $ \varphi (t ,\omega) >0 $ for all $t>t_{1}(\omega)$. Since $\vert \sin (x)\vert \geqslant \frac{2}{\pi}x=:\lambda x$ for $ x\in [-\frac{\pi}{2},\frac{\pi}{2}]$, then for $x\in [-\frac{\pi}{2},\frac{\pi}{2}]$, 
\begin{equation*}
0\leqslant \vert \Phi_{h}(x)\vert\leqslant e^{-\lambda h}\vert x\vert.
\end{equation*}
Thus,
\begin{eqnarray}
\vert \varphi (t+T)\vert &\leqslant & \vert \varphi (t+T)-\Phi_{T}(\varphi (t))\vert+ \vert \Phi_T (\varphi (t))\vert \notag\\
&\leqslant & \delta (t)+ e^{-\lambda T}\vert \varphi (t)\vert. \label{rec}
\end{eqnarray}
Set $y_{k}:= \vert \varphi (kT)\vert$, $\delta_{k}=\delta (kT)$ and $\gamma := e^{-\lambda T} <1$. Thanks to (\ref{rec}), we have
\begin{equation} y_{k+1}\leqslant \delta_k +\gamma y_k.
\end{equation}
Hence 
\begin{equation} y_{k+r}\leqslant \gamma^k y_r +\sum_{j=0}^{k-1}\delta_{k+r-j}\gamma^j.
\end{equation}
By assumption i., $\delta_k \gamma^{-k}$ converges to $\infty$. Thus
\begin{equation} y_k =O(\delta_k).
\end{equation}
For $kT\leqslant t\leqslant (k+1)T$, we have
\begin{eqnarray}
\vert \varphi (t)\vert &\leqslant & \vert \Phi_{t-kT}(\varphi (kT)) -\varphi (t)\vert+ \vert \Phi_{t-kT} (\varphi (kT))\vert \notag\\
&\leqslant & \delta_k+ e^{-\lambda (t-kT)}y_k\notag\\
&=& O(\delta_k). 
\end{eqnarray}
Hence, $$\vert \varphi (t)\vert = O(\delta (t)). $$

\textit{Step II:} Let $T>h>0$. From Theorem \ref{CRITasymp}, there exists constants $C,C(T)$ such that
\begin{equation}
\mathbb{P}(\sup_{0\leqslant h\leqslant T}\vert \varphi (t+h)-\Phi_{h}(\varphi (t))\vert>\beta) \leqslant  C\exp(-\frac{(\beta C(T))^2}{n\dot{\alpha} (t)}).
\end{equation}
Defining $\beta_{\gamma,T} (t):= \frac{1}{C(T)}(\dot{\alpha}^{1/2}(t)\log^{\frac{\gamma}{2}} (1+t))$, with $\gamma >1$, we get
\begin{equation} \mathbb{P}(\sup_{0\leqslant h\leqslant T}\vert \varphi (t+h)-\Phi_{h}(\varphi (t))\vert>\beta_{\gamma ,T}(t))\leqslant C\exp(-\log^{\gamma} (1+t))
\end{equation}
Since 
\begin{equation}\int_0^\infty \exp(-\log^{\gamma} (1+t))dt<\infty
\end{equation}
we deduce by the Borel-Cantelli Lemma that almost-surely
\begin{eqnarray} \sup_{0\leqslant h\leqslant T}\vert \varphi (t+h)-\Phi_{h}(\varphi (t))\vert &=& O(\beta_{\gamma,T} (t)) \notag\\
&=& O(\dot{\alpha}^{1/2}(t)\log^{\frac{\gamma}{2}} (t))\notag\\
&=& O(t^{-1/4}\log^{\frac{\gamma}{2}} (t)). 
\end{eqnarray}

\textit{Step III:} Since $\alpha(t)= \sqrt{2t}$, it is clear that assumptions i. and ii. are satisfied. So, from Steps I and II,
\begin{equation}\vert \varphi_{t}\vert = O(t^{-1/4}\log^{\frac{\gamma}{2}} (t)).
\end{equation}
Therefore, 
\begin{equation} \vert Y_{t}\vert = O(t^{-1/2}\log^\gamma (t)).
\end{equation}
Consequently, 
\begin{equation} \vert Z_{t}\vert = O(t^{-1/2}\log^\gamma (t)).
\end{equation}
and so
\begin{equation} \vert \Theta_{t}\vert = O(t^{-1}\log^\gamma (t)).
\end{equation}
\section{Conclusion}
The motivating model of this work was the real-valued self-attracting diffusion 
\begin{equation*}
dX_{t}=\sigma dW_{t}+a\int_{0}^{t}\sin(X_{t}-X_{s})dsdt,\qquad X_{0}=0.
\end{equation*} 
Seeing it as an angle, it turned out that the almost sure convergence of $X_t$ was an immediate consequence of the more general diffusion on the $n-$dimensional unit sphere $\mathbb{S}^n$
\begin{equation*}
dX(t)=\sigma  dW_{t}(X(t))-a\int_{0}^{t}\nabla_{\mathbb{S}^n}V_{X_s}(X_t) dsdt,\qquad X(0)=x\in\mathbb{S}^n
\end{equation*} 
with $V_y(x)=\langle x,y\rangle$.

It would now be interesting to study the self-reinforced diffusion $$dX_{t}=\sigma dW_{t}+\sum_{k=1}^{n}ka_{k}\int_{0}^{t}\sin(k(X_{t}-X_{s}))dsdt, $$
where the coefficient $a_{k}\neq 0$ are such that $\sum_{k=1}^{n}k^{2}a_{k}<0$.

Because $\sum_{k=1}^{n}k^{2}a_{k}=(\sum_{k=1}^{n}ka_{k}\sin(k.))'(0)$ and that it has to play a more and more important role if $(X_{t})_{t}$ localizes, it sounds reasonable to formulate the following conjecture:
\begin{conjecture} Let $(X_{t})_{t\geqslant 0}$ be the solution to the SDE $$dX_{t}=\sigma dW_{t}+\sum_{k=1}^{n}ka_{k}\int_{0}^{t}\sin(k(X_{t}-X_{s}))dsdt,\: X_{0}=x. $$
If $\sum_{k=1}^{n}k^{2}a_{k}<0$ (resp. $\sum_{k=1}^{n}k^{2}a_{k}>0$), then $X_{t}$ converges almost-surely (resp. $\limsup_{t}X_{t}>\liminf_{t}X_{t}$).
\end{conjecture} 

\appendix
\section{Almost sure convergence for the linear case}
Since $\sin (x)\sim x$ on a small neighbourhood of $0$, the aim to this appendix is to show that the rate of convergence that we obtain is quasi-optimal.
\begin{proposition}\label{linearise0} Let $X_{t}$ be the solution of the SDE
\begin{equation}\label{linearise} dX_{t}=-\lambda X_{t} dt +\sqrt{\varepsilon (t)}dB_{t},
\end{equation}
with initial condition $X_{0}=x$. $(B_{t})_{t}$ stands for a real Brownian motion and $\lambda >0 $. Assume that $\varepsilon (.)$ is a non-increasing positive continuous function such that $\lim_{t\rightarrow\infty}\varepsilon (t)=0$. 

Set $\sigma_{t}^2=e^{-2\lambda t}\int_0^t e^{2\lambda s}\varepsilon (s)ds$. 
\begin{enumerate} 
\item If $\int_0^t e^{2\lambda s}\varepsilon (s)ds<\infty$, then $\vert X_{t}\vert =O(e^{-\lambda t})$\\
\item If $\int_0^t e^{2\lambda s}\varepsilon (s)ds=\infty$, then $\vert X_{t}\vert =O(\sigma_{t}\sqrt{\log (t)})$. In particular,
\begin{enumerate} 
\item If $\varepsilon (t)=O(e^{-2\alpha t})$ with $\alpha <\lambda$, then $\sigma_{t}^2 =O(e^{-2\alpha t})$\\
\item If $\varepsilon (t)=O(t^{-\alpha})$, $\alpha >0$, then $\sigma_t^2= O(\varepsilon (t))$
\end{enumerate}
\end{enumerate}
\end{proposition}   
\begin{proof}
We assume without loss of generality that $\varepsilon(0)<\infty$. The solution of Equation (\ref{linearise}) is 
\begin{equation}\label{solLin} X_{t}= \exp(-\lambda t)(x+\int_{0}^{t}\exp(\lambda s )\sqrt{\varepsilon (s)}dB_{s}=:\exp(-\lambda t)(x+M_{t}).
\end{equation}
The quadratic variation of $M_{t}$ is then
\begin{equation*} \langle M\rangle_{t} = \int_{0}^{t}\exp(2\lambda s)\varepsilon (s)ds.
\end{equation*}

For the first assertion of the proposition,  we then have that $t\mapsto M_{t}$ is bounded. The conclusion follows from (\ref{solLin}).

Concerning the second statement, we have, by the Dubins-Schwarz Theorem (see Theorem 4.6 in \cite{K-S}) with the law of Iterated Logarithm for Brownian motion (see Theorem 9.23, Chapter 2 in \cite{K-S}), 
\begin{equation}\label{estim0}\limsup_{t\rightarrow\infty}\frac{\vert M_{t}\vert }{\sqrt{2\langle M\rangle_{t} \log\log \langle M\rangle_{t}}}=1 
\end{equation}
almost surely. Thus $\frac{\vert M_{t}\vert }{\sqrt{2\langle M\rangle_{t} \log\log \langle M\rangle_{t}}}$ is bounded.\\
Since $\langle M\rangle_{t}\leqslant \frac{\varepsilon (0)}{2\lambda}(\exp(2\lambda t)-1)$, we have 
\begin{equation}\label{estim1} \log\log \langle M\rangle_{t} = O(\log(t)).
\end{equation}
Thus, 
\begin{eqnarray}\label{estim3}
\sqrt{2\langle M\rangle_{t} \log\log \langle M\rangle_{t}} \exp(-\lambda t)&=&\sqrt{\exp(-2\lambda t)\langle M\rangle_{t}}\sqrt{2\log\log \langle M\rangle_{t}}\notag\\
&= & O(\sigma_{t}\sqrt{\log(t)}).
\end{eqnarray}

Part (2.(a)) follows from the definition of $\sigma_{t}^2$. For part (2.(b)), we have
\begin{eqnarray*} \exp(-2\lambda t)\int_{0}^{t}\exp(2\lambda s)\varepsilon (s)ds &=& \exp(-2\lambda t)\int_{0}^{t-\sqrt{t}}\exp(2\lambda s)\varepsilon (s)ds\\
& & +\exp(-2\lambda t)\int_{t-\sqrt{t}}^{t}\exp(2\lambda s)\varepsilon (s)ds\\
&\leqslant & \frac{\varepsilon (0)}{2\lambda}\exp(-2\lambda t)(\exp(2\lambda t-2\lambda \sqrt{t})-1)\\
& & +\frac{1}{2\lambda}\exp(-2\lambda t)\varepsilon (t-\sqrt{t})\exp(2\lambda t)\\
& & -\frac{1}{2\lambda}\exp(-2\lambda t)\varepsilon (t-\sqrt{t})\exp(2\lambda t-2\lambda \sqrt{t}).
\end{eqnarray*}
Thus 
\begin{eqnarray*} \exp(-2\lambda t)\langle M\rangle_{t} &\leqslant& \frac{\varepsilon (0)}{2\lambda}(\exp(-2\lambda \sqrt{t})-\exp(-2\lambda t))\\
& & +\frac{1}{2\lambda}\varepsilon (t-\sqrt{t})(1-\exp(-2\lambda \sqrt{t})).
\end{eqnarray*}
Because for all $\beta \geqslant 0$, $(t-\sqrt{t})^\beta$ is equivalent to $t^\beta$ when $t\rightarrow \infty$, we obtain the existence of a constant $C$ such that 
\begin{equation}\label{estim2} \exp(-2\lambda t)\langle M\rangle_{t}\leqslant C^{2} \varepsilon (t)
\end{equation}
for $t$ large enough. \\
\end{proof}
\begin{remark} To ensure almost sure convergence to $0$, the maximal noise intensity has to be $\varepsilon(t)=O(1/\log(t)^{\alpha}),$ with $\alpha >1$. 
\end{remark}
A direct consequence of Proposition \ref{linearise0}.$2.ii$ is
\begin{corollary} If $\varepsilon (t)=(1+t)^{-\alpha}$ with $\alpha>0$, then the solution $(X_t)_{t\geqslant 0}$ of (\ref{linearise}) satisfies $\vert X_{t}\vert =O(t^{-\frac{\alpha}{2}}\sqrt{\log(t)})$. 
\end{corollary}
 \section*{Acknoledgement}
I thank my PhD advisor Michel Bena\"im for the very useful related discussions and comments on previous versions of this work, Ioana Ciotir for her proofreading of the first version and Olivier Raimond for stimulating discussion and his suggestion to consider the case of the sphere in all dimensions. Finally, I thank the anonymous referee for all the comments made on a previous version.

\end{document}